\documentclass[reqno]{llncs}
\usepackage[utf8]{inputenc}
\usepackage{amsmath}
\usepackage{amsfonts}
\usepackage{amssymb}
\usepackage[svgnames]{xcolor}
\usepackage{colortbl}
\usepackage{iba-algo}

\makeatletter
\let\bfseries=\undefined
\DeclareRobustCommand\bfseries
{\not@math@alphabet\bfseries\mathbf
  \boldmath\fontseries\bfdefault\selectfont}
\makeatother

\newcommand{\Z}{{\mathbb Z}}

\newcommand{\R}{\mathbb R}

\DeclareMathOperator{\diam}{diam}
\DeclareMathOperator{\dist}{dist}

\DeclareMathOperator{\supp}{supp}

\def\Graver{{\mathcal G}}
\def\GLP{{\mathcal C}_{\leq}}

\def\Circuits{{\mathcal C}}

\usepackage{enumerate}
\def\ve#1{\mathchoice{\mbox{\boldmath$\displaystyle\bf#1$}}
{\mbox{\boldmath$\textstyle\bf#1$}}
{\mbox{\boldmath$\scriptstyle\bf#1$}}
{\mbox{\boldmath$\scriptscriptstyle\bf#1$}}}

\newcommand\veb{{\ve b}}
\newcommand\vecc{{\ve c}}

\newcommand\veg{{\ve g}}

\newcommand\vel{{\ve l}}

\newcommand\veu{{\ve u}}
\newcommand\vev{{\ve v}}
\newcommand\vew{{\ve w}}
\newcommand\vex{{\ve x}}
\newcommand\vey{{\ve y}}
\newcommand\vez{{\ve z}}

\newcommand{\eoproof}{\hspace*{\fill} $\square$ \vspace{5pt}}

\usepackage{algorithm}
\usepackage{ifthen}
\usepackage{tikz}
\usepackage{subfigure}
\usepackage{caption}
\usepackage{here}

\newcommand{\T}{{\intercal}} %transpose

\newcommand{\DeclareBracket}[3]{
  \newcommand{#1}[2][]{%
  \ifthenelse%
  {\equal{##1}{}}%
  {\left#2##2\right#3}%
  {\csname ##1l\endcsname#2##2\csname ##1r\endcsname#3}}}
\DeclareBracket\set\{\}

\usepackage[colorlinks=true,breaklinks=true,bookmarks=true,urlcolor=blue,
     citecolor=blue,linkcolor=blue,bookmarksopen=false,draft=false]{hyperref}

\begin{document}

\title{On the circuit diameter of dual transportation polyhedra}

\author{Steffen Borgwardt\inst{1}\and Elisabeth Finhold\inst{2}\and Raymond Hemmecke\inst{3}}

%\thanks{The author was supported by {TopMath}, a graduate program of the {Elite Network of Bavaria} and the {TUM Graduate School}.}

\institute{\email{\href{mailto:borgwardt@ma.tum.de}{borgwardt@ma.tum.de}};
Technische Universit\"at M\"unchen, Germany \and
\email{\href{mailto:finhold@tum.de}{finhold@tum.de}};
Technische Universit\"at M\"unchen, Germany \and
\email{\href{mailto:hemmecke@tum.de}{hemmecke@tum.de}};
Technische Universit\"at M\"unchen, Germany}

\date{\today}

\maketitle

\begin{abstract}
  In this paper we introduce the circuit diameter of polyhedra, which is always bounded from above by the combinatorial diameter. We consider dual transportation polyhedra defined on general bipartite graphs. For complete $M{\times}N$ bipartite graphs the Hirsch bound $(M{-}1)(N{-}1)$ on the \emph{combinatorial diameter} is a known tight bound (Balinski, 1984). For the \emph{circuit diameter} we show the much stronger bound $M{+}N{-}2$ for all dual transportation polyhedra defined on arbitrary bipartite graphs with $M{+}N$ nodes.
\end{abstract}


\section{Introduction.}

Graver bases of matrices $A\in\Z^{d\times n}$ were introduced by Jack Graver in 1975 in his seminal paper \cite{Graver:75} as sets $\Graver(A)$ of vectors that provide optimality certificates for the family of integer linear programs $\min\set{\,\vecc^\T\vez:A\vez=\veb, \vel\leq\vez\leq\veu, \vez\in\Z^n\,}$ that share the problem matrix $A$ but that may differ in the remaining data $\veb,\vecc,\vel,\veu$. This optimality certificate provided by $\Graver(A)$ allows to augment any given feasible solution to optimality via a simple scheme similar to the Simplex method for linear programs: iteratively augment the given solution along Graver basis directions until a solution is reached that cannot be augmented along a direction from $\Graver(A)$. This solution must be optimal. Clearly, the number of augmentation steps needed heavily depends on how one chooses among several applicable augmenting Graver basis directions. 

In the last $20$ years, a lot of progress has been made on the theory of Graver bases. It has been shown that $\Graver(A)$ also provides optimality certificates for the minimization of separable convex objective functions over the lattice points of a polyhedron \cite{MurotaSaitoWeismantel04}, that at most polynomially many (in the binary encoding length of the input data) Graver-best augmentation steps are needed in order to reach an optimal solution \cite{Hemmecke+Onn+Weismantel:oracle}, and that $N$-fold separable-convex integer linear programs can be solved in polynomial time \cite{DeLoera+Hemmecke+Onn+Weismantel:08,hemmecke-koeppe-weismantel:4-block-proximity,hemmecke-onn-romanchuk:nfold-cubic}. For a more thorough introduction to the theory of Graver bases and for more references on this topic we refer the interested reader to the books \cite{DHKbook,onn:nonlinear-discrete-monograph}.

Note that the notion of a Graver basis can be extended to the continuous setting of linear programs. Here, the  \emph{circuits} or \emph{elementary vectors} $\Circuits(A)$ of $A\in\Z^{d\times n}$ provide a universal optimality certificate similarly as the Graver basis $\Graver(A)$ does for the integer setting. All results readily translate from the integer linear to the linear setting and the proofs are often much simpler.

Very recently, it has been shown in \cite{DeLoera+Hemmecke+Lee:14} that for integer linear programs and for linear programs one needs at most $|\Graver(A)|$ respectively $|\Circuits(A)|$ many steepest-descent Graver basis augmentation steps. This surprising bound does not depend on $\veb$, $\vecc$, $\vel$ and $\veu$ and readily implies that $N$-fold (integer) linear programs can be solved in \emph{strongly} polynomial time. This raises the natural question of how many circuit augmentation steps are needed with a ``perfect'' selection rule? Progress on this question may lead to a strongly polynomial-time algorithm for the solution of general linear programs via circuit augmentations, which would solve a long-standing open question on the complexity of LPs. The search for a best selection rule leads us to a notion similar to the \emph{combinatorial diameter} of a polyhedron, which gives a lower bound for the number of steps needed by the Simplex method to solve an LP. 

In this paper, we introduce the notion of \emph{circuit diameter} of a polyhedron as the maximum number of (maximum length) steps along circuit directions that are needed to go from any vertex of the polyhedron to any other vertex of the polyhedron. From the definition of the circuits it will follow directly that the circuit diameter of a polyhedron is bounded from above by the combinatorial diameter and thus it is natural to ask, whether the \emph{Hirsch bound} (which has been disproved to bound the combinatorial diameter in general \cite{klee67,santos11}) always bounds the circuit diameter of a polyhedron. 

\begin{conjecture}[Circuit diameter bound]\label{Conjecture: Circuit diameter bound}
  For any $n$-dimensional polyhedron with $f$ facets the circuit diameter is bounded above by $f-n$.
\end{conjecture}

It is an immediate interesting open question whether the counterexamples to the Hirsch conjecture \cite{klee67,santos11} give rise to counterexamples to our Conjecture \ref{Conjecture: Circuit diameter bound} or not. 

To bound the combinatorial diameter of a polyhedron it suffices to consider generic polyhedra, as by perturbation any polyhedron can be turned into a generic polyhedron, whose diameter is at least as big as the one of the original polyhedron. It is not clear whether the same is true for the circuit diameter, see the second example presented in the next section.

In this paper, we consider dual transportation polyhedra defined on general bipartite graphs. For 
complete $M{\times}N$ bipartite graphs the Hirsch bound $(M{-}1)(N{-}1)$ on the \emph{combinatorial diameter} has been already proved and shown to be tight \cite{balinski84}. For the \emph{circuit diameter} we show the much stronger bound $M{+}N{-}2$ for all dual transportation polyhedra defined on bipartite graphs with $M+N$ nodes. This shows that there are families of polyhedra whose circuit diameter is much smaller than their combinatorial diameter, which gives hope that an augmentation algorithm along circuit directions could have a much better complexity to solve LPs than the Simplex method.

\section{Circuit distance and circuit diameter}

The \emph{circuits} or \emph{elementary vectors} of a matrix $A\in\Z^{d\times n}$ are the support-minimal elements in $\ker(A)\setminus\set{\,\ve 0\,}$, normalized to (coprime) integer components. Clearly, there are only finitely many such vectors. It can be shown that the set of circuits consists exactly of all edge directions of $\set{\,\vez:A\vez=\veb,\vez\geq\ve 0\,}$ for varying $\veb$ \cite{SturmfelsThomas97}. This also implies that the set of circuits of $A$ provides a universal optimality certificate for linear programs $\min\set{\,\vecc^\T\vez:A\vez=\veb,\vez\geq\ve 0\,}$ for any choice of $\veb$ and $\vecc$; similarly as the Graver basis of $A$ does for the integer setting.

In analogy to this, we define for the linear program $\min\set{\,\vecc^\T\vez:A\vez\leq\veb\,}$ the set of circuits $\GLP(A)$ as the collection of all edge directions of $\set{\,\vez:A\vez\leq\veb\,}$ for varying $\veb$. (Note that the matrix $A\in\Z^{d\times n}$ should have full row rank $n$ for the polytope to have vertices and edges.) It is not hard to show that these edge directions are given by those vectors $\vez\in\R^n\setminus\set{\,\ve 0\,}$ for which $\supp(A\vez)$ is inclusion-minimal among all supports $\supp(A\vex)$, $\vex\in\R^n\setminus\set{\,\ve 0\,}$. It is also not hard to show that $\GLP(A)$ provides augmenting directions to any non-optimal solution of $\min\set{\,\vecc^\T\vez:A\vez\leq\veb\,}$ for any choice of $\veb$ and $\vecc$.

One should note that for a linear program, augmentation along circuit directions is a generalization of the Simplex method: While in the Simplex method one walks only along the $1$-skeleton/edges (so in particular on the boundary) of the polyhedron, the circuit steps are allowed to go through the interior of the polyhedron (along \emph{potential} edge directions). While \cite{Hemmecke+Onn+Weismantel:oracle} states that there is a selection strategy such that only polynomially many circuit augmentation steps are needed to reach an optimal solution (a fact that is still unresolved for the Simplex method), it is still open how to implement this greedy-type augmentation oracle in polynomial time.

Inspired by the surprising bound of at most $|\Circuits(A)|$ circuit augmentations \cite{DeLoera+Hemmecke+Lee:14}, one may wonder if there is a selection strategy such that only a strongly polynomial number (that depends only on $d$ and $n$) of augmentation steps is needed to reach an optimal solution. We will not answer this fundamental question here, but introduce and turn to an intimately related problem. For this, let us define the notions of \emph{circuit distance} and \emph{circuit diameter}. 

\begin{definition}
  Let $P=\set{\,\vez:A\vez\leq\veb\,}$ be a polyhedron. For two vertices $\vev^{(1)},\vev^{(2)}$ of $P$, we call a sequence $\vev^{(1)}=\vey^{(0)},\ldots,\vey^{(k)}=\vev^{(2)}$ a \emph{circuit walk of length $k$} if for all $i=0,\ldots,k-1$ we have
\begin{enumerate}
\item $\vey^{(i)}\in P$,
\item $\vey^{(i+1)}-\vey^{(i)}=\alpha_i\veg^{(i)}$ for some $\veg^{(i)}\in\GLP(A)$  and $\alpha_i>0$, and
\item $\vey^{(i)}+\alpha\veg^{(i)}$ is infeasible for all $\alpha>\alpha_i$. 
\end{enumerate}  
 The \emph{circuit distance} $\dist_\Circuits(\vev^{(1)},\vev^{(2)})$ from $\vev^{(1)}$ to $\vev^{(2)}$ then is the minimum length of a circuit walk from $\vev^{(1)}$ to $\vev^{(2)}$. 
 The \emph{circuit diameter} $\diam_\Circuits(P)$ of $P$ is the maximum circuit distance between any two vertices of $P$.
\end{definition} 

It should be noted that a circuit walk is not necessarily reversible, so we may have $\dist_\Circuits(\vev^{(1)},\vev^{(2)})\neq \dist_\Circuits(\vev^{(2)},\vev^{(1)})$. The following example demonstrates that this can indeed happen.

\begin{example}
  Consider the polyhedron $P=\set{\,\vez:A\vez\leq\veb\,}$ given by 
\[
	A=
		\left(\begin{array}{rr}
			-1 & 0\\ -1 & 1 \\ 0 & 1 \\ 1 & 1 \\ 1 & -1 \\ -1 & 1
		\end{array} \right)
	\quad \text{and} \quad
	\veb=
		\left(\begin{array}{r}
			0 \\ 1 \\2\\4\\6\\0
		\end{array} \right).
\]
$P$ is a two-dimensional polytope with six vertices, whose circuits are given by 
\[
  \GLP(A)=\left\{
 	  \pm \left(\begin{array}{r} 1\\0 \end{array} \right), 
 	  \pm \left(\begin{array}{r} 0\\1 \end{array} \right),  
 	  \pm \left(\begin{array}{r} 1\\1 \end{array} \right), 
 	  \pm \left(\begin{array}{r} 1\\-1 \end{array} \right)
 	\right\}.
\]
\noindent Now let us have a look at the circuit distances $\dist_\Circuits(\vev^{(1)},\vev^{(4)})$ and $\dist_\Circuits(\vev^{(4)},\vev^{(1)})$:
 
\definecolor{polytopeColor}{gray}{0.9}
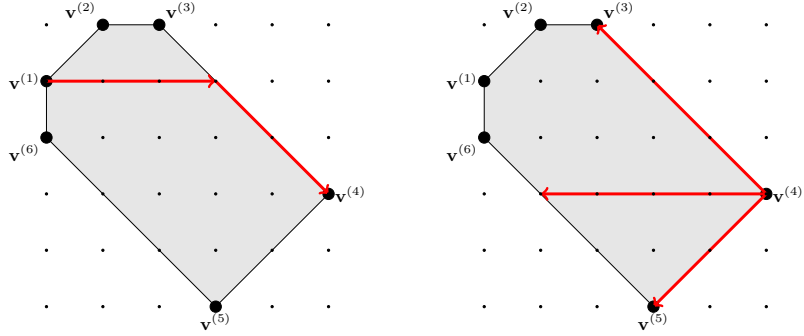
	\begin{figure}[H]
		\centering
\scalebox{.75}{
			\begin{tikzpicture}[scale=1]
\useasboundingbox (-1,-4) rectangle (6,3) ;
					\coordinate (v1) at (0,1);
					\coordinate (v2) at (1,2);
					\coordinate (v3) at (2,2);
					\coordinate (v4) at (5,-1);
					\coordinate (v5) at (3,-3);
					\coordinate (v6) at (0,0);
					\draw[black, fill= polytopeColor] (v1)--(v2)--(v3)--(v4)--(v5)--(v6)--(v1);					
					\draw [fill, black] (v1) circle [radius=0.1];
					\draw [fill, black] (v2) circle [radius=0.1];
					\draw [fill, black] (v3) circle [radius=0.1];
					\draw [fill, black] (v4) circle [radius=0.1];
					\draw [fill, black] (v5) circle [radius=0.1];
					\draw [fill, black] (v6) circle [radius=0.1];
					\node[left] at (v1) {$\vev^{(1)}$};
					\node[above left] at (v2) {$\vev^{(2)}$};
					\node[above right] at (v3) {$\vev^{(3)}$};
					\node[right] at (v4) {$\vev^{(4)}$};
					\node[below] at (v5) {$\vev^{(5)}$};
					\node[below left] at (v6) {$\vev^{(6)}$};
					\draw[line width= 1.5, red, ->] (v1)--(3,1);
					\draw[line width= 1.5, red, ->] (3,1)--(v4);
					\foreach \x in {0,1,...,5}{
        		\foreach \y in {-3,-2,...,2}{
				 			\draw [fill, black] (\x,\y) circle [radius=0.02];
						}
					}
			\end{tikzpicture} 	
\qquad
			\begin{tikzpicture}[scale=1]
\useasboundingbox (-1,-4) rectangle (6,3) ;
					\coordinate (v1) at (0,1);
					\coordinate (v2) at (1,2);
					\coordinate (v3) at (2,2);
					\coordinate (v4) at (5,-1);
					\coordinate (v5) at (3,-3);
					\coordinate (v6) at (0,0);
					\draw[black, fill= polytopeColor] (v1)--(v2)--(v3)--(v4)--(v5)--(v6)--(v1);					
					\draw [fill, black] (v1) circle [radius=0.1];
					\draw [fill, black] (v2) circle [radius=0.1];
					\draw [fill, black] (v3) circle [radius=0.1];
					\draw [fill, black] (v4) circle [radius=0.1];
					\draw [fill, black] (v5) circle [radius=0.1];
					\draw [fill, black] (v6) circle [radius=0.1];
					\node[left] at (v1) {$\vev^{(1)}$};
					\node[above left] at (v2) {$\vev^{(2)}$};
					\node[above right] at (v3) {$\vev^{(3)}$};
					\node[right] at (v4) {$\vev^{(4)}$};
					\node[below] at (v5) {$\vev^{(5)}$};
					\node[below left] at (v6) {$\vev^{(6)}$};
					\draw[line width= 1.5, red, ->] (v4)--(v3);
					\draw[line width= 1.5, red, ->] (v4)--(1,-1);
					\draw[line width= 1.5, red, ->] (v4)--(v5);
					\foreach \x in {0,1,...,5}{
        		\foreach \y in {-3,-2,...,2}{
				 			\draw [fill, black] (\x,\y) circle [radius=0.02];
						}
					}
			\end{tikzpicture}
} 	
	\end{figure}
  We have $\dist_\Circuits(\vev^{(1)},\vev^{(4)})=2$, but $\dist_\Circuits(\vev^{(4)},\vev^{(1)})=3$. No matter which circuit direction we choose for a first step starting at $\vev^{(4)}$, we cannot go to $\vev^{(1)}$ with only one more step.
\eoproof
\end{example}

The following example demonstrates that perturbing the right-hand side vector may not change the combinatorial structure of the polyhedron while changing the circuit diameter. Note that both polyhedra possess the same set of edge directions/circuits.

\begin{example}
  Consider the polyhedron $\tilde P=\set{\,\vez:A\vez\leq\tilde\veb\,}$ given by
\[
	A=
		\left(\begin{array}{rr}
			-1 & 0\\ -1 & 1 \\ 0 & 1 \\ 1 & 1 \\ 1 & -1 \\ -1 & 1
		\end{array} \right)
	\quad \text{and} \quad
	\tilde\veb=
		\left(\begin{array}{r}
			0 \\ 1 \\2\\4\\4\\0
		\end{array} \right).
\]
$P$ and $\tilde P$ have the same combinatorial structure:

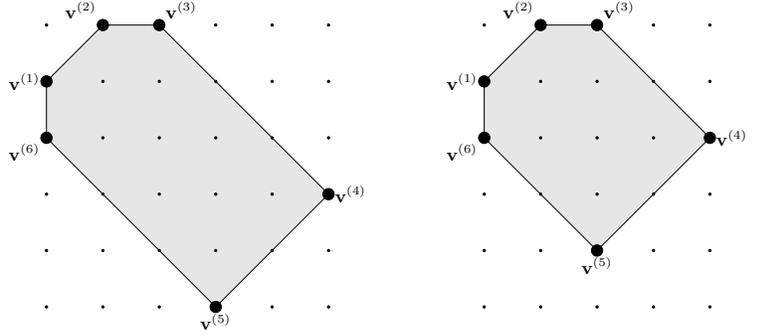
\definecolor{polytopeColor}{gray}{0.9}
	\begin{figure}[H]
		\centering

\scalebox{.75}{
			\begin{tikzpicture}[scale=1]
\useasboundingbox (-1,-4) rectangle (6,3) ;
					\coordinate (v1) at (0,1);
					\coordinate (v2) at (1,2);
					\coordinate (v3) at (2,2);
					\coordinate (v4) at (5,-1);
					\coordinate (v5) at (3,-3);
					\coordinate (v6) at (0,0);
					\draw[black, fill= polytopeColor] (v1)--(v2)--(v3)--(v4)--(v5)--(v6)--(v1);					
					\draw [fill, black] (v1) circle [radius=0.1];
					\draw [fill, black] (v2) circle [radius=0.1];
					\draw [fill, black] (v3) circle [radius=0.1];
					\draw [fill, black] (v4) circle [radius=0.1];
					\draw [fill, black] (v5) circle [radius=0.1];
					\draw [fill, black] (v6) circle [radius=0.1];
					\node[left] at (v1) {$\vev^{(1)}$};
					\node[above left] at (v2) {$\vev^{(2)}$};
					\node[above right] at (v3) {$\vev^{(3)}$};
					\node[right] at (v4) {$\vev^{(4)}$};
					\node[below] at (v5) {$\vev^{(5)}$};
					\node[below left] at (v6) {$\vev^{(6)}$};
					\foreach \x in {0,1,...,5}{
        		\foreach \y in {-3,-2,...,2}{
				 			\draw [fill, black] (\x,\y) circle [radius=0.02];
						}
					}
			\end{tikzpicture} 	
\qquad
			\begin{tikzpicture}[scale=1]
\useasboundingbox (-1,-4) rectangle (6,3) ;
					\coordinate (v1) at (0,1);
					\coordinate (v2) at (1,2);
					\coordinate (v3) at (2,2);
					\coordinate (v4) at (4,0);
					\coordinate (v5) at (2,-2);
					\coordinate (v6) at (0,0);
					\draw[black, fill= polytopeColor] (v1)--(v2)--(v3)--(v4)--(v5)--(v6)--(v1);					
					\draw [fill, black] (v1) circle [radius=0.1];
					\draw [fill, black] (v2) circle [radius=0.1];
					\draw [fill, black] (v3) circle [radius=0.1];
					\draw [fill, black] (v4) circle [radius=0.1];
					\draw [fill, black] (v5) circle [radius=0.1];
					\draw [fill, black] (v6) circle [radius=0.1];
					\node[left] at (v1) {$\vev^{(1)}$};
					\node[above left] at (v2) {$\vev^{(2)}$};
					\node[above right] at (v3) {$\vev^{(3)}$};
					\node[right] at (v4) {$\vev^{(4)}$};
					\node[below] at (v5) {$\vev^{(5)}$};
					\node[below left] at (v6) {$\vev^{(6)}$};
					\foreach \x in {0,1,...,5}{
        		\foreach \y in {-3,-2,...,2}{
				 			\draw [fill, black] (\x,\y) circle [radius=0.02];
						}
					}
			\end{tikzpicture} 	
} 	
\end{figure}
It is not hard to check that $\diam_\Circuits(P)=3$ while $\diam_\Circuits(\tilde P)=2$. This indicates that perturbing the right-hand side may have effects on the circuit diameter that are hard to predict.
\eoproof
\end{example}

Clearly, the circuit diameter of a polyhedron is at most as large as the combinatorial diameter of the polyhedron, as a walk along the $1$-skeleton/edges of the polyhedron is a circuit walk. This raises the natural question whether the well-known Hirsch conjecture holds for the circuit diameter in place of the combinatorial diameter. Recall that there are counterexamples to the Hirsch conjecture bounding the combinatorial diameter of polyhedra and polytopes \cite{klee67,santos11}.

In the following section we consider the circuit diameter of dual transportation polyhedra defined on bipartite graphs $G=(V,E)$ (that are not necessarily complete). We show that their circuit diameter is bounded from above by $|V|-2$.

\section{Dual transportation polyhedra}

Let $G=(V,E)$ be a connected bipartite graph on node sets $V_1=\set{\,0,\ldots,M-1\,}$ and $V_2=\set{\,M,\ldots,M+N-1\,}$ with edges $E$ having one endpoint in $V_1$ and one endpoint in $V_2$. A dual transportation polyhedron associated to $G$ is given by some vector $\vecc\in\R^{|E|}$ via
\[
  P_{G,\vecc}=\set{\,\veu\in\R^{M+N}:-u_a+u_b\leq c_{ab}\ \forall\ a\in V_1, b\in V_2 \text{ and } ab\in E, u_0=0\,}.
\]
As is standard, we put $u_0=0$ to make $P_{G,\vecc}$ pointed. When we consider the circuit diameter of a specific polyhedron $P_{G,\vecc}$, we may assume that none of the inequalities $-u_a+u_b\leq c_{ab}$ is redundant (otherwise remove such an edge $ab$ from $G$, leaving the polyhedron the same but making the set of circuits smaller and thus the circuit diameter potentially bigger).

Moreover, we may assume that $P_{G,\vecc}$ is generic, although this merely simplifies the presentation. The vertices of $P_{G,\vecc}$ are determined by sets of inequalities $-u_a+u_b\leq c_{ab}$ that become tight. For $\veu\in P_{G,\vecc}$, we denote by $G(\veu)$ the graph with nodes $V$ and with edges $ab\in E$ for which $-u_a+u_b\leq c_{ab}$ is tight. For a vertex $\veu$ of $P_{G,\vecc}$, $G(\veu)$ is a spanning subgraph of $G$ which is always a spanning tree of $G$ if $P_{G,\vecc}$ is generic. (This can be proved on similar lines as in \cite{balinski84} for the complete bipartite graph.) For our proofs it will be enough to know that for each vertex of $P_{G,\vecc}$ there is a spanning tree of $G$ with edges corresponding to the inequalities $-u_a+u_b\leq c_{ab}$ that are tight at the vertex. This uniquely determines the vertex $\veu$, since we normalized $u_0=0$.

As in \cite{balinski84}, the possible edge directions of $P_{G,\vecc}$ can be described as follows: Let $R,S\subseteq V$ be connected nonempty node sets with $R\cup S=V$ and $R\cap S=\emptyset$. W.l.o.g., we may assume $0\in R$. Then the vector $\veg\in\R^{M+N}$ with
\begin{equation}\label{Eq: Construction of circuit direction from R and S.}
  g_i=\left\{\begin{array}{ll}0, & \text{if } i\in R,\\ 1, & \text{if } i\in S,\end{array}\right.
\end{equation}
is an edge direction of $P_{G,\vecc}$ for some right-hand side $\vecc$. In fact, it can be shown that these are all possible edge directions and hence they constitute the set of circuits, $\Circuits_G$, associated to the matrix defining the polyhedron $P_{G,\vecc}$. 

We are ready to present and prove the core part of our main result.

\begin{lemma}\label{Thm: circuit diameter bounded by V-1}
  The circuit diameter $\diam_\Circuits(P_{G,\vecc})$ of $P_{G,\vecc}$ is bounded from above by $|V|-1$.
\end{lemma}

\begin{proof}
  Let $\veu^{(1)}$ and $\veu^{(2)}$ be two vertices of $P_{G,\vecc}$ given by the spanning trees $T_1=G(\veu^{(1)})$ and $T_2=G(\veu^{(2)})$ of $G$. We will show how to construct a circuit walk $\veu^{(1)}=\vey^{(0)},\ldots,\vey^{(k)}=\veu^{(2)}$, such that $G(\vey^{(i)})$ has at least $i$ edges in common with $T_2$. This immediately implies $k\leq |V|-1$ which proves the claim.

It should be noted that the subgraphs $G(\vey^{(i)})$ may not be connected, since our circuit walk possibly goes through the interior of $P_{G,\vecc}$ along (potential) edge directions and thus may enter the interior of higher-dimensional faces of $P_{G,\vecc}$.

Given the feasible point $\vey^{(i)}\neq\veu^{(2)}$, let $C=G(V(C),E(C))$ be the connected component of $(V,E(G(\vey^{(i)}))\cap E(T_2))$ containing the node $0$. Possibly, $C$ consists only of the node $0$. As $\vey^{(i)}\neq\veu^{(2)}$, we must have $C\neq T_2$ and thus there is some node $s\in V$ which is not in $C$, but which is connected to $C$ directly via some edge $rs$ in $T_2$. We now construct an edge direction $\veg$ from $\Circuits_G$ such that $\vey^{(i+1)}:=\vey^{(i)}+\alpha\veg$ arises from a maximal length step along $\veg$ and such that $(V,E(G(\vey^{(i+1)}))\cap E(T_2))$ contains $C$ and the edge $rs$ from $T_2$. Starting from $\vey^{(0)}$ and repeating this process iteratively, we see that $G(\vey^{(i)})$ has at least $i$ edges in common with $T_2$, implying the result.

To construct $\veg$, we need to define $R,S\subseteq V$ that describe the edge direction from $\Circuits_G$. W.l.o.g.\ we will assume that $s\in V_2$. The case $s\in V_1$ works analogously by merely switching the roles of $V_1$ and $V_2$ and hence by switching the roles of $\epsilon\veg$ and $-\epsilon\veg$ below.

\begin{itemize}
  \item[(a)] All nodes from $C$ are assigned to $R$.
  \item[(b)] All nodes from $V_2\setminus\set{\,s\,}$ which are connected to $C$ by an edge in $E$, are assigned to $R$.
  \item[(c)] All nodes $t\in V\setminus R$ that are connected to $s$ by a path in $G$ are assigned to $S$.
  \item[(d)] All remaining nodes are assigned to $R$.
\end{itemize}

As $G$ is connected, this construction leads to sets $R$ and $S$ that are nonempty and that define connected components of $G$ that are connected by an edge $rs\in E$. Hence, $R$ and $S$ define an element $\veg\in\Circuits_G$ via Equation (\ref{Eq: Construction of circuit direction from R and S.}). We wish to include the edge $rs$ into our graph, that is, we wish to make the inequality $-u_r+u_s\leq c_{rs}$ tight at $\vey^{(i+1)}$, that is, we wish to increase the component $y^{(i)}_s$ ($s\in V_2$ and as $s\not\in V(C)$). Hence we \emph{add} $\epsilon\veg$ to $\vey^{(i)}$. (If $s\in V_1$, we \emph{subtract} $\epsilon\veg$ from $\vey^{(i)}$.) We choose as $\epsilon$ the smallest nonnegative number such that an inequality $-u_a+u_b\leq c_{ab}$ with $a\in R$ and $b\in S$ becomes tight. Note that $\epsilon=0$ is not excluded, but we show that this will never happen. In fact, we show that the edge $ab$ (on which $-u_a+u_b\leq c_{ab}$ becomes tight) is exactly the edge $rs$ that we wish to include.

Assume now on the contrary that $ab\neq rs$. Note that by construction at steps (b) and (c) we must have $b=s$, as all edges from $R$ to $S\cap V_2$ have $s$ as common end point and these are exactly the edges on which an inequality may become tight when walking along direction $\veg\in\Circuits_G$. Hence we must have $a\neq r$. As $G(\vey^{(i)}+\epsilon\veg)$ and $T_2$ coincide on the edges in $C$ and since $0\in V(C)$, $\vey^{(i+1)}:=\vey^{(i)}+\epsilon\veg$ and $\veu^{(2)}$ agree in their components in $V(C)$, that is, $u^{(2)}_c=y^{(i+1)}_c$ for all $c\in V(C)$. Since $\vey^{(i+1)}\in P_{G,\vecc}$ and since $as\in E(G(\vey^{(i+1)}))$ and $rs\notin E(G(\vey^{(i+1)}))$, we have 
\[
  -u^{(2)}_a+y^{(i+1)}_s=c_{as}\text{ but } -u^{(2)}_r+y^{(i+1)}_s<c_{rs}.
\] 
On the other hand, since $\veu^{(2)}\in P_{G,\vecc}$ and since $as\notin E(T_2)$ and $rs\in E(T_2)$, we have 
\[
  -u^{(2)}_a+u^{(2)}_s<c_{as}\text{ but } -u^{(2)}_r+u^{(2)}_s=c_{rs}.
\]
From $-u^{(2)}_a+y^{(i+1)}_s=c_{as}$ and $-u^{(2)}_a+u^{(2)}_s<c_{as}$ we conclude $y^{(i+1)}_s>u^{(2)}_s$, whereas $-u^{(2)}_r+y^{(i+1)}_s<c_{rs}$ and $-u^{(2)}_r+u^{(2)}_s=c_{rs}$ imply $y^{(i+1)}_s<u^{(2)}_s$. This contradiction shows $a=r$ and the claim is proved. \eoproof
\end{proof}

We can strengthen this result by observing the following fact on the vertices of $P_{G,\vecc}$.

\begin{lemma}
  Let $\veu^{(1)}$ and $\veu^{(2)}$ be two vertices of $P_{G,\vecc}$ given by the spanning trees $T_1$ and $T_2$ of $G$. Then $E(T_1)\cap E(T_2)\neq\emptyset$.
\end{lemma}

\begin{proof}
As we can translate $P_{G,\vecc}$, we may assume w.l.o.g.{} that $\veu^{(1)}=\ve 0$. Clearly, this mere shift does not change any structure, in particular, it does not change $T_1=G(\veu^{(1)})$ and $T_2=G(\veu^{(2)})$. For better readability, let us denote the components of $\veu^{(i)}$, $i=1,2$, belonging to $V_1$ and $V_2$, respectively, by $\vev^{(i)}$ and $\vew^{(i)}$. Now assume that $E(T_1)\cap E(T_2)=\emptyset$. 

As $T_1$ is connected, there must be an edge $(v_1,w_1)\in T_1$. As $(v_1,w_1)\notin T_2$, we have $-v_1^{(2)}+w_1^{(2)}<-v_1^{(1)}+w_1^{(1)}=0$ and hence $w_1^{(2)}<v_1^{(2)}$.

As $T_2$ is connected, there must be an edge $(v_2,w_1)\in T_2$. As $(v_2,w_1)\notin T_1$, we must have $0=-v_2^{(1)}+w_1^{(1)}<-v_2^{(2)}+w_1^{(2)}$ and hence $v_2^{(2)}<w_1^{(2)}$.

Again, as $T_1$ is connected, there must be an edge $(v_2,w_2)\in T_1$. As $(v_2,w_2)\notin T_2$, we have $-v_2^{(2)}+w_2^{(2)}<-v_2^{(1)}+w_2^{(1)}=0$ and hence $w_2^{(2)}<v_2^{(2)}$. 
	\begin{figure}[H]
		\centering
			\begin{tikzpicture}[scale=0.35]
					\coordinate (v1) at (0,8);
					\coordinate (v2) at (0,6);
					\coordinate (v3) at (0,4);
					\coordinate (v4) at (0,2);
					\coordinate (w1) at (5,8);
					\coordinate (w2) at (5,6);
					\coordinate (w3) at (5,4);
					\coordinate (w4) at (5,2);
					
					\draw [fill, black] (v1) circle [radius=0.08];
					\draw [fill, black] (v2) circle [radius=0.08];
					\draw [fill, black] (v3) circle [radius=0.08];
					\draw [fill, black] (v4) circle [radius=0.08];
					\draw [fill, black] (w1) circle [radius=0.08];
					\draw [fill, black] (w2) circle [radius=0.08];
					\draw [fill, black] (w3) circle [radius=0.08];
					\draw [fill, black] (w4) circle [radius=0.08];

					\node[left] at (v1) {$v_1$};
					\node[left] at (v2) {$v_2$};
					\node[left] at (v3) {$v_3$};
					\node[left] at (v4) {$v_4$};
					\node[right] at (w1) {$w_1$};	
					\node[right] at (w2) {$w_2$};
					\node[right] at (w3) {$w_3$};
					\node[right] at (w4) {$w_4$};
	
				  \draw (v1)--(w1);
					\draw[dotted] (v2)--(w1);
					\draw (v2)--(w2);
					\draw[dotted] (v3)--(w2);
					\draw (v3)--(w3);
					\draw[dotted] (v4)--(w3);
					\draw (v4)--(w4);
					\draw[dotted] (v2)--(w4);
					
					\draw (12,8)--(14,8);
					\node[right] at (14,8) {edge $\in T_1$};
					\draw[dotted]  (12,7)--(14,7);
					\node[right]  at (14,7) {edge $\in T_2$};
			\end{tikzpicture} 	
	\end{figure}
Continuing like this, we create a path with edges alternately from $T_1\setminus T_2$ and $T_2\setminus T_1$. As there are only finitely many nodes, eventually some $v_i$ (or $w_j$) is selected a second time and we close a cycle. But then we have that 
\[
  v_i^{(2)}>w_i^{(2)}>v_{i+1}^{(2)}>\ldots>v_k^{(2)}=v_i^{(2)} 
\]
(or $w_j^{(2)}>\ldots>w_j^{(2)}$), a contradiction. Hence we must have $E(T_1)\cap E(T_2)\neq\emptyset$. \eoproof
\end{proof}

This now implies the following strengthening of Lemma \ref{Thm: circuit diameter bounded by V-1}.

\begin{theorem}
  The circuit diameter $\diam_\Circuits(P_{G,\vecc})$ of $P_{G,\vecc}$ is bounded from above by $|V|-2$.
\end{theorem}

\begin{proof}
  The proof is analogous to the proof of Lemma \ref{Thm: circuit diameter bounded by V-1}. We merely have to observe that w.l.o.g.{} we may assume that the edge that is common to $G(\veu^{(1)})$ and to $G(\veu^{(2)})$ has $0$ as one of its endpoints. So 
%  this edge belongs to $E(G(\veu^{(1)}))\cap E(G(\veu^{(2)}))$ and 
we only have to add at most $|V|-2$ edges in at most $|V|-2$ steps. \eoproof
\end{proof}

\section*{Acknowledgments} 

The second author gratefully acknowledges the support from the graduate program TopMath of the Elite Network of Bavaria and the TopMath Graduate Center of TUM Graduate School at Technische Universit\"at M\"unchen.

\bibliography{biblioAugmentation}

\begin{thebibliography}{10}

\bibitem{balinski84}
M.~L. Balinski.
\newblock The {H}irsch conjecture for dual transportation polyhedra.
\newblock {\em Mathematics of Operations Research}, 9(4):629--633, 1984.

\bibitem{DHKbook}
J.~A. De~Loera, R.~Hemmecke, and M.~K{\"o}ppe.
\newblock {\em Algebraic and geometric ideas in the theory of discrete
  optimization}, volume~14 of {\em MOS-SIAM Series on Optimization}.
\newblock Society for Industrial and Applied Mathematics (SIAM), Philadelphia,
  PA, 2013.

\bibitem{DeLoera+Hemmecke+Lee:14}
J.~A. De~Loera, R.~Hemmecke, and J.~Lee.
\newblock Augmentation algorithms for linear and integer linear programming.
\newblock in preparation, 2014.

\bibitem{DeLoera+Hemmecke+Onn+Weismantel:08}
J.~A. De~Loera, R.~Hemmecke, S.~Onn, and R.~Weismantel.
\newblock {$N$}-fold integer programming.
\newblock {\em Discrete Optimization}, 5(2):231--241, 2008.
\newblock In Memory of George B. Dantzig.

\bibitem{Graver:75}
J.~E. Graver.
\newblock On the foundation of linear and integer programming {I}.
\newblock {\em Mathematical Programming}, 9:207--226, 1975.

\bibitem{hemmecke-koeppe-weismantel:4-block-proximity}
R.~Hemmecke, M.~K{\"o}ppe, and R.~Weismantel.
\newblock Graver basis and proximity techniques for block-structured separable
  convex integer minimization problems.
\newblock eprint arXiv:1207.1149, to appear in {\sl Mathematical Programming},
  2012.

\bibitem{hemmecke-onn-romanchuk:nfold-cubic}
R.~Hemmecke, S.~Onn, and L.~Romanchuk.
\newblock {$N$}-fold integer programming in cubic time.
\newblock {\em Mathematical Programming}, 137:325--341, 2013.

\bibitem{Hemmecke+Onn+Weismantel:oracle}
R.~Hemmecke, S.~Onn, and R.~Weismantel.
\newblock A polynomial oracle-time algorithm for convex integer minimization.
\newblock {\em Mathematical Programming}, 126:97--117, 2011.

\bibitem{klee67}
V.~Klee and D.~W. Walkup.
\newblock The $d$-step conjecture for polyhedra of dimension $d < 6$.
\newblock {\em Acta Mathematica}, 133:53--78, 1967.

\bibitem{MurotaSaitoWeismantel04}
K.~Murota, H.~Saito, and R.~Weismantel.
\newblock Optimality criterion for a class of nonlinear integer programs.
\newblock {\em Operations Research Letters}, 32:468--472, 2004.

\bibitem{onn:nonlinear-discrete-monograph}
S.~Onn.
\newblock {\em Convex Discrete Optimization}.
\newblock Zurich Lectures in Advanced Mathematics. European Mathematical
  Society, 2010.

\bibitem{santos11}
F.~Santos.
\newblock A counterexample to the {H}irsch conjecture.
\newblock {\em Annals of Mathematics (Princeton Univ. and Institute for
  Advanced Study)}, 176(1):383--412, 2011.

\bibitem{SturmfelsThomas97}
B.~Sturmfels and R.~R. Thomas.
\newblock Variation of cost functions in integer programming.
\newblock {\em Mathematical Programming}, 77:357--388, 1997.

\end{thebibliography}
\bibliographystyle{plain}

\end{document}